\documentclass[a4paper,11pt]{article}

\usepackage{amssymb}
\usepackage{graphicx}
\usepackage{wrapfig}
\usepackage{caption}
\usepackage{subcaption}

\usepackage{amsfonts}
\usepackage{amsmath}
\usepackage{amssymb}
\usepackage{color}
\usepackage{graphicx}
\usepackage{xspace}
\usepackage[matrix,arrow]{xy}
\usepackage{mathrsfs,eucal,enumerate}
\usepackage{amsthm,amscd}
\usepackage{mathtools}
\usepackage{epsfig,psfrag}
\usepackage{verbatim}
\usepackage{xr}

\newtheorem{Th}{Theorem}
\newtheorem*{Prop}{Proposition}
\newtheorem{Lm}{Lemma}

\newtheorem*{Addth}{Addendum to Theorem~\ref{ThBE}}
\theoremstyle{definition}
\newtheorem{Def}{Definition}
\newtheorem{Rem}{Remark}

\renewcommand{\ge}{\geqslant}
\renewcommand{\le}{\leqslant}

\newcommand{\zcz}{z\ii\C[[z\ii]]}

\newcommand{\al}{\alpha}
\newcommand{\be}{\beta}
\newcommand{\de}{\delta}
\newcommand{\De}{\Delta}
\newcommand{\ga}{\gamma}
\newcommand{\Ga}{\Gamma}
\newcommand{\sig}{\sigma}
\renewcommand{\th}{\theta}

\newcommand{\ph}{\varphi}
\newcommand{\eps}{\varepsilon}
\newcommand{\ka}{\kappa}
\newcommand{\la}{\lambda}
\newcommand{\La}{\Lambda}
\newcommand{\om}{\omega}

\newcommand{\ze}{\zeta}

\newcommand{\dd}{{\mathrm d}}      
\newcommand{\ee}{{\mathrm e}}      

\newcommand{\id}{\operatorname{id}}
\newcommand{\ID}{\operatorname{Id}}

\newcommand{\defeq}{\coloneqq} 
\newcommand{\col}{\colon\thinspace}          

\newcommand{\demi}{\tfrac{1}{2}}
\newcommand{\ii}{^{-1}}
\newcommand{\ti}{\tilde}
\newcommand{\ens}{\enspace}

\newcommand{\ie}{\emph{i.e.}\ }
\newcommand{\eg}{\emph{e.g.}\ }

\newcommand{\cf}{\emph{cf.}}

\DeclareMathOperator{\IM}{Im}         
 \DeclareMathOperator{\RE}{Re}        
\DeclareMathOperator{\dist}{dist}     

\newcommand{\cont}{\operatorname{cont}}

\newcommand{\I}{\mathrm{i}}

\newcommand{\C}{\mathbb{C}}      
\newcommand{\N}{\mathbb{N}}        

\newcommand{\R}{\mathbb{R}}        
\newcommand{\Z}{\mathbb{Z}}        

\newcommand{\cL}{ {\cal L }}

\newcommand{\Ha}{ {\cal H }}

\newcommand{\gA}{\mathscr A}       
\newcommand{\gB}{\mathscr B}       
\newcommand{\gE}{\mathscr E}       
\newcommand{\gR}{\mathscr R}       

\newcommand{\beglabel}[1]{\begin{equation}	\label{#1}}
\newcommand{\elabel}{\end{equation}}

\DeclarePairedDelimiter\abs{\lvert}{\rvert}%
\DeclarePairedDelimiter\norm{\lVert}{\rVert}%

\newcommand{\up}{^\textrm{up}}
\newcommand{\low}{^\textrm{low}}
\newcommand{\uplow}{^\textrm{up/low}}

\newcommand{\Clog}{\C_{\log}}

\newcommand{\Al}{\operatorname{\gA^\Ga_\om}}
\newcommand{\hAl}{\operatorname{\hat\gA^\Ga_\om}}

\newcommand{\QR}{Q_R}

\newcommand{\RL}{\gR_{\eps,L}}
\newcommand{\tRL}{\ti\gR_{\eps,L}}

\newcommand{\VL}[1]{V_{\eps,L;#1}}

\newcommand{\NL}[2]{\norm{#1}_{\eps,L;#2}}

\newcommand{\bd}[1]{\big[#1\big]_d}
\newcommand{\cd}[1]{\big\{#1\big\}_d}

\DeclareMathOperator{\End}{End}
\DeclareMathOperator{\Span}{Span}

\begin{document}

\thispagestyle{empty}
\begin{center}
{\bf \LARGE
On the resurgent approach\\[1.5ex] to \'Ecalle-Voronin's invariants}\\[3ex]

Artem Dudko, David Sauzin
\end{center}
\begin{abstract}
  Given a holomorphic germ at the origin of~$\C$ with a simple
  parabolic fixed point, the Fatou coordinates have a common
  asymptotic expansion whose formal Borel transform is resurgent.
  We show how to use \'Ecalle's alien operators to study the
  singularities in the Borel plane and relate them to the horn maps,
  providing each of \'Ecalle-Voronin's invariants in the form of a
  convergent numerical series.
  The proofs are original and self-contained, with ordinary Borel
  summability as the only prerequisite.
\medskip

\end{abstract}

\section{Alien operators for simple $2\pi\I\Z$-resurgent series}

We first present an extension of the classical Borel-Laplace summation
theory \cite{Ram}, \cite{Mal_cours}, \cite{Sau_cours}, to be applied to the formal Fatou coordinate of a simple
parabolic germ in next section.

We denote by~$\Clog$ the Riemann surface of the logarithm viewed as
the universal cover of~$\C^*$ with base-point at~$1$.
For $\be_1, \be_2 \in\C$ with $\RE\be_1, \RE\be_2 >-1$
and $\hat\phi_1 \in \ze^{\be_1}\C\{\ze\}$,
$\hat\phi_2 \in \ze^{\be_2}\C\{\ze\}$,
we extend the usual convolution and define
$\hat\phi_1*\hat\phi_2 \in \ze^{\be_1+\be_2+1}\C\{\ze\}$ by
\[
\hat\phi_1*\hat\phi_2(\ze) \defeq
\int_0^1
\hat\phi_1\big( (1-t)\ze \big) \hat\phi_2(t\ze)
\ze\,\dd t
\quad \text{for $\ze\in\Clog$ with $\abs{\ze}$ small enough.}
\]
For $\be\in\C$ with $\RE\be>-1$, we extend the classical formal Borel transform
by defining
\[
\gB \col
\ti\phi = \sum_{n\geq0} c_n z^{-n-\be-1} \in z^{-\be-1}\C[[z\ii]]
\mapsto
\hat\phi \defeq \sum_{n\geq0} c_n \tfrac{\ze^{n+\be}}{\Ga(n+\be+1)}
\in \ze^{\be}\C[[\ze]].
\]
Observe that if $\gB\ti\phi_1 = \hat\phi_1 \in \ze^{\be_1} \C\{\ze\}$
and $\gB\ti\phi_2 = \hat\phi_2 \in \ze^{\be_2} \C\{\ze\}$
then $\gB(\ti\phi_1 \ti\phi_2) =\hat\phi_1*\hat\phi_2$.


\begin{Def} \mbox{}\vspace{-1ex}

\begin{itemize}
\item[--]
  Given $\be\in \C$ with $\RE\be>-1$, if $\hat\phi \in
  \ze^\be\C\{\ze\}$ converges for any~$\ze$ on~$\Clog$ with $\abs{\ze}<2\pi$ and extends analytically
  along any path of $\C\setminus 2\pi\I\Z$ issuing from~$1$,
  then $\hat\phi$ and $\ti\phi = \gB\ii\hat\phi$ are said to be \emph{$2\pi\I\Z$-resurgent}.
\item[--]
Given $\om \in 2\pi\I\Z$ and
a path $\Ga$ in $\C\setminus 2\pi\I\Z$ going from~$1$ to $\om+1$,
if $\hat\phi$ is $2\pi\I\Z$-resurgent and if its analytic continuation along~$\Ga$,
which is a holomorphic germ $\cont_\Ga\hat\phi$ at $\om+1$, takes the form
\beglabel{eqsimplsing}
\cont_\Ga\hat\phi(\om+\ze) = \frac{c}{2\pi\I\ze} +
\hat\chi(\ze) \frac{\log\ze}{2\pi\I}
+ \hat R(\ze)
\quad \text{for $\ze>0$},
\elabel
where $c\in\C$ and $\hat\chi,\hat R \in \C\{\ze\}$,
then $\hat\phi$ and $\ti\phi = \gB\ii\hat\phi$ are said to be
\emph{$(\om,\Ga)$-simple}.
\item[--]
In the above situation, we set
$\hAl\ti\phi \defeq \hat\chi$, which is the monodromy of
$\ze \mapsto \cont_\Ga\hat\phi(\om+\ze)$ around~$0$,
and $\Al\ti\phi \defeq c + \gB\ii\hat\chi \in \C[[z\ii]]$.
The operator~$\Al$ thus defined on the space of all $(\om,\Ga)$-simple
formal series is called the \emph{alien operator associated
with $(\om,\Ga)$}.
\end{itemize}
\end{Def}


\noindent
(Observe that $\hAl\ti\phi$ is itself $2\pi\I\Z$-resurgent because
$\cont_\Ga\hat\phi(\om+\ze)$ extends analytically
along any path of $\C\setminus 2\pi\I\Z$.)


\begin{Def} Let $\hat\phi$ be $2\pi\I\Z$-resurgent.
We say that $\hat\phi$ is of \emph{finite exponential type in
    non-vertical directions} if,
for any path~$\ga$ of $\C\setminus 2\pi\I\Z$ going from~$1$ to $\ze_*\in \I\R$
and any $\delta_0 \in (0,\pi/2)$,
there exist $C_0,R_0>0$ such that
\beglabel{ineqleftright}
\abs*{ \cont_\ga\hat\phi\big( \ze_* + t\,\ee^{\I\th} \big) } \leq
C_0\,\ee^{R_0 t}
\quad \text{for all $t\geq0$ and
$\th \in [-\de_0,\de_0] \cup [\pi-\de_0,\pi+\de_0]$.}
\elabel
\end{Def}


We omit the proof of the following technical statement.
\begin{Lm}\label{Lmzalphares}
Let $\ti\phi_0\in \zcz$ be $2\pi\I\Z$-resurgent with $\hat\phi_0\defeq
\gB\ti\phi_0$ of finite exponential type in non-vertical directions.
Let $\ti\phi\defeq z^{-\al}\ti\phi_0\in z^{-\be-1}\C[[z^{-1}]]$, where
$\al,\be\in \C$ and $\RE\be>-1$.
Then $\hat\phi\defeq \gB\ti\phi \in \ze^\be\C\{\ze\}$.
Moreover, for any simple curve $\ga \col (0,+\infty) \to \C\setminus
2\pi\I\Z$ of the form $\ga(t) = t u_0$ for $t<t_0$
and $\ga(t) = c + t u_1$ for $t>t_1$, with $\abs{u_0} = \abs{u_1} =
1$, $u_1\neq\pm\I$, $c\in\C$ and $0 < t_0 < t_1$,
and for any lift~$\ti\ga$ of~$\ga$ in~$\Clog$,
the germ~$\hat\phi$ admits analytic continuation along~$\ti\ga$ and
%
%
\beglabel{eqintegralzalpha}
\int_{\ti\ga} \ee^{-z\ze} \hat\phi(\ze) \,\dd\ze =
z^{-\al} \int_\ga e^{-z\ze} \hat\phi_0(\ze) \,\dd\ze
\elabel
for all $z\in\Clog$ with
$\arg z\in(-\theta-\tfrac{\pi}{2},-\theta+\tfrac{\pi}{2})$
and $\RE(ze^{\I\theta})$ large enough,
where $\th\in\R$ is the argument of~$u_1$ which determines the
asymptotic direction of~$\ti\ga$.

In particular, $\cL^\th\hat\phi(z) \defeq \int_0^{ \ee^{\I\th} \infty}
\ee^{-z\ze}\hat\phi(\ze)\,\dd\ze$ is well defined and coincides with
$z^{-\al}\cL^\th\hat\phi_0(z)$ for $z \in \Clog$ as above.
\end{Lm}

%


\begin{Lm}   \label{LmHankel}
Let $\th\in(-\frac{\pi}{2},\frac{\pi}{2}) \cup
(\frac{\pi}{2},\frac{3\pi}{2})$ and denote by~$\Ha_\th$ a $\th$-rotated
Hankel contour, \ie a contour in~$\Clog$ which goes along a ray from
$\ee^{\I(\th-2\pi)}\infty$ to $\ee^{\I(\th-2\pi)}\eps$ (with $0<\eps<\pi$), circles
counterclockwise~$0$ and then follows the ray from $\ee^{\I\th}\eps$
to $\ee^{\I\th}\infty$.
Suppose that $\hat\phi$ is $(\om,\Ga)$-simple and $\ze \mapsto
\cont_\Ga\hat\phi(\om+\ze)$ has at most exponential growth
along~$\Ha_\th$, then
\[
\int_{\Ha_\th} \ee^{-z\ze} \cont_{\Ga}\hat\phi(\om+\ze)
\, \dd\ze =
c + \cL^\th\hat\chi(z),
\]
where $\Al\gB\ii\hat\phi = c + \gB\ii\hat\chi$, 
$\arg z\in(-\theta-\tfrac{\pi}{2},-\theta+\tfrac{\pi}{2})$ 
and $\RE(ze^{\I\theta})>>0$.
\end{Lm}

\begin{proof}
The term $\frac{c}{2\pi\I\ze}$ contributes~$c$, 
the remainder contributes
$\int_{\ee^{\I\th}\eps}^{\ee^{\I\th}\infty} \ee^{-z\ze}\hat\chi(\ze)\,\dd\ze
+ O(\eps\abs{\ln\eps})$.
\end{proof}


\begin{Lm}    \label{LmCommutAl}
If $b_0 \in \C\{z\ii\}$ and $\ti\phi$ is $2\pi\I\Z$-resurgent, then
the formal series $b_0\ti\phi$ and $C_{\id+b_0}\ti\phi \defeq
\ti\phi\circ(\id+b_0)$ are $2\pi\I\Z$-resurgent.
If moreover $\ti\phi$ is $(\om,\Ga)$-simple, then they are also
$(\om,\Ga)$-simple, with
\[
\Al(b_0\ti\phi) = b_0 \Al\ti\phi,
\qquad
\Al C_{\id+b_0}\ti\phi =
\ee^{-\om b_0} C_{\id+b_0}\Al\ti\phi.
\]
\end{Lm}


\begin{proof}[Idea of the proof]
Start with $b_0 \in z\ii\C\{z\ii\}$,
hence $\hat b_0(\ze)$ is entire;
%
%
$\ti\psi_1 \defeq {b_0} \ti\phi$ and
$\ti\psi_2 \defeq C_{\id+{b_0}}\ti\phi - \ti\phi$
have Borel images
\beglabel{eqhatpsij}
\hat\psi_j(\ze) = \int_0^\ze K_j(\xi,\ze) \hat\phi(\xi) \, \dd\xi
\quad \text{for $\ze\in\Clog$ with $\abs{\ze} < 2\pi$,}
\qquad j=1,2,
\elabel
where $K_1(\xi,\ze) = \hat b_0(\ze-\xi)$ and
$K_2(\xi,\ze) = \sum_{k\geq1} \frac{(-\xi)^k}{k!} \hat b_0^{*k}(\ze-\xi)$
are holomorphic in $\C\times\C$
(\cf\ the proof of Lemma~2 in \cite{DSun}).
The analytic continuation of~$\hat\psi_1$ or~$\hat\psi_2$ along a
path $\ga \col [0,1]\to\C\setminus2\pi\I\Z$ is thus given by the same
integral as~\eqref{eqhatpsij}, but integrating
on the concatenation $[0,\ga(0)] + \ga + [\ga(1),\ze]$.
It is then possible to analyze the singularities of~$\hat\psi_1$
and~$\hat\psi_2$ at~$\om$.
We omit the details.
%
%
\end{proof}



The above is inspired by \'Ecalle's resurgence theory \cite{Eca81}.
In fact, $2\pi\I\Z$-resurgent series are stable under multiplication
(see \cite{Eca81}, \cite{Sau13} or \cite{Sau14}),
and so are the $(\om,\Ga)$-simple ones, but we shall not require these
facts in this article,
nor the far-reaching extension of the framework which allows to define
alien operators or alien derivations in much more general situations.

\section{The singular structure of the formal Borel transform of the
  formal Fatou coordinate}

We use the same notations as in \cite{DSun}:
$f = (\id+b)\circ(\id+1) = z + 1 - \rho z\ii + O(z^{-2})$ is a simple parabolic germ at~$\infty$;
under the change of unknown $v(z) = z + \rho\log z + \ph(z)$
the equation $v\circ f = v+1$ is transformed into
\beglabel{eqphbis}
C_{\id-1} \ph = C_{\id+b} \ph + b_*,
\qquad b_*(z) \defeq b(z) + \rho\log\tfrac{1+z\ii b(z)}{1-z\ii} \in z^{-2}\C\{z\ii\}
\elabel
which has a unique solution $\ti\ph\in\zcz$,
shown to be $2\pi\I\Z$-resurgent and Borel summable.
Setting $\hat\ph\defeq\gB\ti\ph$, one gets two normalized Fatou coordinates
$v_*^\pm(z) = z + \rho \log z + \cL^\pm\hat\ph(z)$, 
where $\cL^+\hat\ph(z)$ and~$\cL^-\hat\ph(z)$ are Laplace transforms
along~$\R^+$ and~$\R^-$, holomorphic for $\arg z \in (-\pi,\pi)$ and
$\arg z \in (0,2\pi)$,
and a pair of normalized lifted horn maps 
$h_*\uplow = v_*^+ \circ (v_*^-)\ii$
(or \'Ecalle-Voronin modulus---see \cite{Eca81}, \cite{Vor81},
\cite{DSun}):
\begin{align*}
& h_*\up(Z) = Z + \sum_{m\geq1} A_m \ee^{2\pi\I m Z} 
\quad \text{for $\IM  Z >>0$}, \\[1ex]
\quad
& h_*\low(Z) = Z -2\pi\I\rho + \sum_{m\geq1} A_{-m} \ee^{-2\pi\I m Z}
\quad \text{for $\IM  Z <<0$}.
\end{align*}
We are interested in the \'Ecalle-Voronin invariants $A_m$, $m\in\Z^*$.


Let $\om \in 2\pi\I\Z^*$ and $\Ga$ be a path in $\C\setminus2\pi\I\Z$
going from~$1$ to $\om+1$.
In general $\ti\ph$ is not $(\om,\Ga)$-simple, but we shall directly
prove that, up to a linear combination of monomials,
$z^{\rho\om}\ti\ph$ is $(\om,\Ga)$-simple.
More precisely, we shall prove the $(\om,\Ga)$-simplicity of $\ti\Phi \defeq
z^{\rho\om}\{\ti\ph\}_N$, with the notation
\[
\ti\ph = [\ti\ph]_N + \{\ti\ph\}_N,
\qquad
[\ti\ph]_N \in \Span( z\ii, z^{-2}, \ldots, z^{-N} ),
\quad \{\ti\ph\}_N \in z^{-N-1}\C[[z\ii]],
\]
where $N\geq \max\{ 0, -2\RE\al \}$ is integer, with $\al \defeq -\rho\om$.
Observe that $\ti\Phi \in z^{-\be-1}\C[[z\ii]]$ with $\be \defeq
\al+N$ and $\RE\be \geq 0$.
We shall also compute $\Al\ti\Phi$ and relate it to the lifted horn maps of~$f$.

%
We first introduce an auxiliary sequence of formal series, using the
same operator \[ E \col z^{-2}\C[[z\ii]] \to z\ii\C[[z\ii]] \]
inverse of $C_{\id-1}-\ID$
as in \cite{DSun}, but replacing~$B$ with
\[ B^\om \col \ti\psi \mapsto
\ee^{-\om b_*} C_{\id+b} \ti\psi - \ti\psi =
\ee^{-\om b_*} B\ti\psi + (\ee^{-\om b_*} - 1)\ti\psi. \]


\begin{Prop}
%
%
The formal series $\ti\Psi_k = (E B^\om)^{k} 1 \in z^{-k}\C[[z\ii]]$
yield a formally convergent series $\sum\ti\Psi_k$
and $\sum_{k\geq0}\ti\Psi_k = \ee^{-\om\ti\ph}$.
For each $k\geq1$, $\ti\Psi_k$ is $2\pi\I\Z$-resurgent and, for any
$\eps,L>0$, there exist $C_0,M>0$ such that
\[
\abs{\cont_\ga \gB\ti\Psi_k} \leq C_0 \frac{M^k}{k!},
\qquad \text{for all $k\geq 1$ and $\ga \in \RL$},
\]
where $\RL$ is the set of all paths of length $< L$ issuing from~$1$ and staying at
distance $>\eps$ from $2\pi\I\Z$.
\end{Prop}


\begin{proof}
Let $c^\om \defeq \ee^{-\om b_*} - 1$. 
The estimates are obtained by adapting
the proof of Lemma~2 of \cite{DSun}, with a new
kernel function
$K^\om(\xi,\ze) = \hat c^\om(\ze-\xi) +
\sum_{k\geq1} \frac{(-\xi)^k}{k!} \hat b^{*k}(\ze-\xi)
+ \sum_{k\geq1} \frac{(-\xi)^k}{k!} (\hat c^\om * \hat b^{*k})(\ze-\xi)$.
The formal series $\ti\Psi \defeq \sum_{k\geq0}\ti\Psi_k$ is the
unique solution with constant term~$1$ of
$C_{\id-1}\ti\Psi = \ee^{-\om b_*} C_{\id+b}\ti\Psi$%
---so is $\ee^{-\om\ti\ph}$, as can be seen by exponentiating~\eqref{eqphbis}.
\end{proof}


Since $\RE\be\geq0$,
the operator $C_{\id-1}-\ID$ maps $z^{-\be-1}\C[[z\ii]]$ to
$z^{-\be-2}\C[[z\ii]]$ bijectively; we denote by~$E_\be$ the inverse,
whose Borel counterpart is the multiplication operator
\[ \hat E_\be \col \hat\phi(\ze) \in \ze^{\be+1}\C[[\ze]] \mapsto
\frac{1}{\ee^\ze-1} \hat\phi(\ze) \in \ze^\be\C[[\ze]]. \]
We introduce an operator
\[
B_\al \col \ti\phi \in z^{-\be-1}\C[[z\ii]] \mapsto
c_\al C_{\id+b} \ti\phi - \ti\phi \in z^{-\be-2}\C[[z\ii]],
\]
with
$c_\al \defeq \big( \tfrac{1+z\ii b}{1-z\ii} \big)^{\al} \in 1 +z\ii\C\{z\ii\}$.


\begin{Th}    \label{ThBE}
Let
$b_\al \defeq z^{-\al} (1-z\ii)^{-\al} b_N$,
with
$b_N \defeq b_* - C_{\id-1} [\ti\ph]_N + C_{\id+b} [\ti\ph]_N$.
\begin{itemize}
\item[--]
The formal series $\ti\Phi_k(z) \defeq (E_\be B_\al)^k E_\be b_\al \in
z^{-\be-1}\C[[z\ii]]$, $k\geq0$, are $(\om,\Ga)$-simple,
and so is $z^{-\al}\{\ti\ph\}_N$.
\item[--]
For any $\eps,L>0$, there exist $C_0,\La>0$ with $\La<1$ such
that, for each $\ga \in \RL$,
$\abs{ \cont_\ga \gB\ti\Phi_k } \leq C_0 \La^k$.
Moreover $\hat\Phi \defeq \sum_{k\geq0}\gB\ti\Phi_k$
coincides with the formal Borel transform of $z^{-\al}\{\ti\ph\}_N$.
\item[--]
There exist complex numbers $S^\Ga_{\om,k}$, $k\geq0$, which are $O(\La^k)$ for
some $\La<1$ and satisfy
\begin{align}
\label{eqBE}
\Al \ti\Phi_k & = \sum_{k_1+k_2=k} S^\Ga_{\om,k_1} \ti\Psi_{k_2}
\quad\text{ for each $k\geq0$},\\[1ex]
\label{eqBEb}
\Al(z^{-\al}\{\ti\ph\}_N) & = S^\Ga_\om \, \ee^{-\om\ti\ph}
\quad \text{where} \ens
S^\Ga_\om = \sum_{k\geq0} S^\Ga_{\om,k}.
\end{align}
\end{itemize}
\end{Th}


\begin{Th}    \label{ThEVinv}
For $m \in \Z^*$, let $\Ga_m$ denote the line segment going from~$1$
to~$2\pi\I m +1$.
Then the \'Ecalle-Voronin invariants of~$f$ are given by
\[
A_{-m} = S^{\Ga_m}_{2\pi\I m} \, \ee^{-4\pi^2 m\rho}
\ens \text{if $m>0$,} 
\qquad
A_{-m} = -S^{\Ga_m}_{2\pi\I m}
\ens \text{if $m<0$.} 
\]
\end{Th}
%

\begin{Rem}
When $\al\neq0$, the series of formal series $\sum\ti\Phi_k$ is not
convergent for the topology of the formal convergence
(but it converges for the product topology $z^{-\be-1}\C[[z\ii]]
\simeq \C^{\N}$)
and, unless $-\al\in\N$ or $S_\om^\Ga = 0$, the singularity at~$\om$ of
$\cont_\Ga\gB\ti\ph$ is not ``simple'': it is rather of the form
$\cont_\Ga\gB\ti\ph(\om+\ze) =
\frac{1}{\Ga(-\al)} S_\om^\Ga \ze^{-\al-1}(1+O(\abs{\ze})) +O(1)$
if $\al\notin\Z$
or $\frac{(-1)^\al \al!}{2\pi\I} S_\om^\Ga \ze^{-\al-1}(1+O(\abs{\ze}))$
if $\al\in\N$.
\end{Rem}


\begin{Rem}
Theorem~\ref{ThEVinv} can be extracted from \cite{Eca81}, 
which gives detailed proofs only for the case $\rho=0$,
and also the case $\Ga=\Ga_m$ of~\eqref{eqBEb} with
the equivalent formulation 
$\De^+_{\om_m} \ti\ph = A_m z^{-\rho\om_m}\ee^{-\om_m\ti\ph}$
(``Bridge equation''), where $\om_m=2\pi\I m$.
The name ``resurgence'' evokes the fact that the singular behaviour near~$\om_m$
of the analytic continuation of a germ at~$0$ like~$\hat\Phi$ can be
explicitly expressed in terms of~$\hat\Phi$ itself.
\end{Rem}


The representation of $\hat\Phi = \gB(z^{-\al}\{\ti\ph\}_N)$ as the
convergent series $\sum\hat\Phi_k = \sum \gB \big( (E_\be B_\al)^k E_\be b_\al \big)$ and the
computation of the action of the alien operator~$\Al$ in~\eqref{eqBE}--\eqref{eqBEb} are new.
The point is that
$\Al\ti\Phi_k = S^\Ga_{\om,k} + O(z\ii)$, so Equation~\eqref{eqBE} says that
\[
\cont_\Ga \hat\Phi_k(\om+\ze) =
\frac{S^\Ga_{\om,k}}{2\pi\I\ze} +
\hAl\ti\Phi_k(\ze) \frac{\log\ze}{2\pi\I} +
\text{regular germ}, \quad
\hAl\ti\Phi_k = \sum_{k_1+k_2=k,\ k_2\geq1} S^\Ga_{\om,k_1}
\hat\Psi_{k_2} \in \C\{\ze\}.
\]
The numbers $S^\Ga_{\om,k}$ appear as generalized residua, for which
we can give quite explicit formulas:
%
%
\begin{Addth}
Let $\hat c_* \defeq \gB(c_\al-1)$, which is an entire function, and 
\beglabel{eqdefKal}
K_\al(\xi,\ze) = \hat c_*(\ze-\xi) +
\sum_{k\geq1} \frac{(-\xi)^k}{k!} \hat b^{*k}(\ze-\xi) +
\sum_{k\geq1} \frac{(-\xi)^k}{k!} (\hat c_* * \hat b^{*k})(\ze-\xi),
\elabel
which is holomorphic in $\C\times\C$.
Let~$\ti\Ga$ denote a parametrization of the path obtained by
concatenating the line segment $(0,1]$, the path~$\Ga$ and the line
segment $[\om+1,\om]$.
Then
\begin{align*}
S_{\om,0}^\Ga &= 2\pi\I \cont_{\ti\Ga}\hat b_\al (\om), \\[1ex]
S_{\om,k}^\Ga &= 2\pi\I \int_{\De_{\ti\Ga,k}} 
\cont_{\ti\Ga}\hat b_\al(\xi_1)
\frac{
K_\al(\xi_1,\xi_2)\cdots K_\al(\xi_{k-1},\xi_k)K_\al(\xi_k,\om)
}{
(\ee^{\xi_1}-1) \cdots (\ee^{\xi_{k-1}}-1)(\ee^{\xi_k}-1)
}
\, \dd\xi_1 \wedge \cdots \wedge \dd\xi_k,
\qquad k\ge 1,
\end{align*}
with the notation
$\De_{\ti\Ga,k} \defeq \big\{\, \big( \ti\Ga(s_1),\ldots,\ti\Ga(s_k) \big) 
\mid s_1 \le \cdots \le s_k \,\big\}$ for each positive integer~$k$.
\end{Addth}
%

Observe that, according to~\eqref{eqBEb}, each coefficient~$S_\om^\Ga$, and in particular each
\'Ecalle-Voronin invariant, can be obtained as the convergent series of these
``residua''~$S^\Ga_{\om,k}$.


\begin{proof}[Proof of Addendum to Theorem~\ref{ThBE}]
  The normal convergence of the series~\eqref{eqdefKal} follows easily
  from the estimates available for the convolution of entire functions
  (see inequality~(9) in \cite{DSun}), and the Borel counterpart
  of~$B_\al$ is the integral transform with kernel function~$K_\al$,
  hence
\beglabel{eqBalKal}
\cont_\ga \hat B_\al\hat\phi\big( \ga(s) \big) =
\int_0^s K_\al\big( \ga(\sig), \ga(s) \big) \cont_\ga \hat\phi\big( \ga(\sig) \big)
\ga'(\sig) \, \dd\sig
\elabel
for any $2\pi\I\Z$-resurgent~$\hat\phi$ and 
any path $\ga\col (0,\ell] \to \Clog$ (with any $\ell>0$) whose projection onto~$\C$
avoids $2\pi\I\Z$ and which starts as $\ga(s) = s u_0$ for $s>0$ small
enough (with fixed $u_0\in\C^*$).
Writing $\ti\Phi_k = E_\be (B_\al E_\be)^k b_\al$,
by repeated use of~\eqref{eqBalKal},
we obtain an explicit formula for the analytic
continuation of its Borel transform:
for any~$\ze$ close enough to the endpoint of a path~$\ga$ as above,
\begin{align*}
\cont_\ga\hat\Phi_0(\ze) &= \frac{1}{\ee^\ze-1} \cont_{\ga}\hat b_\al (\ze), \\[1ex]
\cont_\ga\hat\Phi_k(\ze) &= \frac{1}{\ee^\ze-1} \int_{\De_{\ga,k}} 
\cont_{\ga}\hat b_\al(\xi_1)
\frac{
K_\al(\xi_1,\xi_2)\cdots K_\al(\xi_{k-1},\xi_k)K_\al(\xi_k,\ze)
}{
(\ee^{\xi_1}-1) \cdots (\ee^{\xi_{k-1}}-1)(\ee^{\xi_k}-1)
}
\, \dd\xi_1 \wedge \cdots \wedge \dd\xi_k,
\qquad k\ge 1,
\end{align*}
with the notation
$\De_{\ga,k} \defeq \big\{\, \big( \ga(s_1),\ldots,\ga(s_k) \big) 
\mid s_1 \le \cdots \le s_k \,\big\}$.
The conclusion follows.
\end{proof}


\begin{proof}[Theorem~\ref{ThBE} implies Theorem~\ref{ThEVinv}]
For all $m\geq1$ we set $\om_m \defeq 2\pi\I m$ and
$\ti\Phi^{(m)} \defeq z^{-\al_m} \{ \ti\ph \}_{N_m}$
with $\al_m \defeq -\rho\om_m$ and
$N_m \defeq \max\{ 0, -2\RE\al_m \}$.
By \cite{DSun} we know that $\hat\ph$ is of finite exponential type in non-vertical directions;
we fix  $\th^+\in(0,\frac{\pi}{2})$, $\th^-\in(\frac{\pi}{2},\pi)$ and
$\QR\low \defeq \big\{\, z \in \Clog \mid
\arg z\in (-\th^+ -\tfrac{\pi}{2}, -\th^- +\tfrac{\pi}{2}), \;
\RE( z \,\ee^{\I\th^+} ) > R, \; \RE( z \,\ee^{\I\th^-} ) > R \,\big\}$,
with $R$ large enough.
Identifying $z\in \QR\low$ with its projection in~$\C$,
we have $\IM z < 0$ and the branches of $\log$ used in~$v_*^+$
and~$v_*^-$ differ by $-2\pi\I$;
for $m_0\geq1$, deforming the contour of integration (see
Figure~\ref{figdeformHank}), we get
\begin{multline*}
v_*^+(z)-v_*^-(z) = -2\pi\I\rho + (\cL^+-\cL^-)\hat\ph(z) \\[1ex]
= -2\pi\I\rho +
\Bigg(\sum_{m=1}^{m_0}\int_{\ga_m}+\int_{\ti\ga_{m_0}}\Bigg)
\,\ee^{-z\ze} \hat\ph(\ze) \,\dd\ze
= -2\pi\I\rho + \sum_{m=1}^{m_0} I_m
+ O(\ee^{(2\pi m_0+\pi)\IM  z}),
\end{multline*}
with $I_m = \int_{\ga_m}
\ee^{-z\ze} \cont_{\Ga_m}\hat\ph(\ze) \, \dd\ze
= \int_{\ga_m}
\ee^{-z\ze} \cont_{\Ga_m} \{\hat\ph(\ze)\}_{N_m} \, \dd\ze$.
Since each~$\ga_m$ can be expressed as the difference of two paths to
which Lemma \ref{Lmzalphares} applies, we have
\[
I_m = z^{\al_m} \int_{\ga_m} \ee^{-z\ze} \cont_{\Ga_m}
\hat\Phi^{(m)}(\ze) \, \dd\ze.
\]
Denoting by $\Ha^-$ a $\th^-$-rotated Hankel contour, we get
$I_m = z^{\al_m}\, \ee^{-\om_m z}
\int_{\Ha^-} \ee^{-z\ze} \cont_{\Ga_m}\hat\Phi^{(m)}(\om_m+\ze) \,\dd\ze
= S^{\Ga_m}_{\om_m} \,\ee^{\al_m\log z -\om_m z -\om_m \cL^-\hat\ph}$
by Lemma~\ref{LmHankel} and~\eqref{eqBEb}
(Borel-Laplace summation commutes with exponentiation---see \eg \cite{Sau14}).
%
\begin{figure}
\begin{center}
\epsfig{file=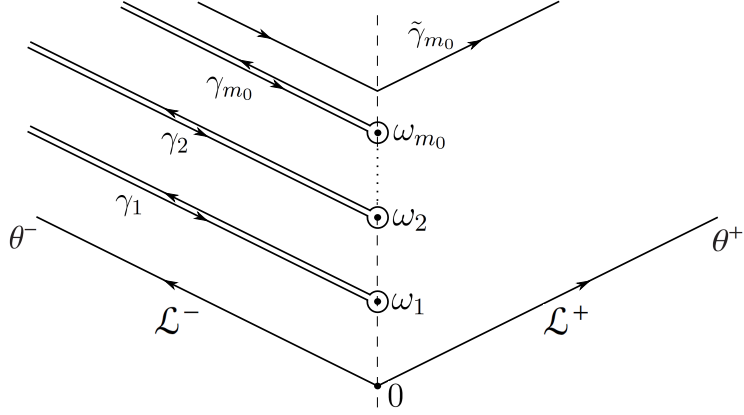,height=5.4cm,angle = 0}
\end{center}
\vspace{-.5cm}
\caption{\label{figdeformHank} Deformation of the contour.}
\end{figure}
%
Thus
\[
(h_*\low-\id)\circ v_*^- = v_*^+-v_*^- = -2\pi\I\rho +
\sum_{m=1}^{m_0} S^{\Ga_m}_{\om_m} \, \ee^{-\om_m ( -2\pi\I\rho + v_*^- )} +
O(\ee^{(2\pi m_0+\pi)\IM  z})
\quad \text{on~$\QR\low$,}
\]
whence
$A_{-m} = S^{\Ga_m}_{\om_m}\, \ee^{2\pi\I\rho\om_m}$
for $m>0$.
The case $m<0$ is similar, except that,
on the domain $\QR\up$ defined analogously with $-\pi < \th^- < -\frac{\pi}{2} < \th^+ < 0$,
the two branches of $\log$ used in~$v_*^\pm$ match
and the orientation of the paths differs.
\end{proof}


\begin{proof}[Proof of Theorem~\ref{ThBE}]
First observe that
$b_\al \in z^{-\be-2}\C\{z\ii\}$ because
$b_*$ and $[\ti\ph]_N \in \C\{z\ii\}$ hence $b_N \in \C\{z\ii\}$,
and
\[
b_N = C_{\id-1} \{\ti\ph\}_N - C_{\id+b} \{\ti\ph\}_N \in z^{-N-2}\C[[z\ii]].
\]
Thus the formal series $\ti\Phi_k$ are well defined in $z^{-\be-1}\C[[z\ii]]$;
they are $2\pi\I\Z$-resurgent because this property is preserved
by~$E_\be$ and~$B_\al$ (Lemma~\ref{LmCommutAl}) and we start with
$\gB\ti\Phi_0 = \hat E_\be \hat b_\al =
\frac{\hat b_\al(\ze)}{\ee^\ze-1}$ meromorphic on~$\Clog$.
Their $(\om,\Ga)$-simplicity and the existence of contants
$S_{\om,k}^\Ga$ such that~\eqref{eqBE} holds
result by induction on~$k$ (using the fact that $B^\om\ti\Psi_k \in
z^{-2}\C[[z\ii]]$ for all~$k$) from
%
%
\begin{Lm}   \label{LmCommutAlBal}
If $\ti\phi$ is $(\om,\Ga)$-simple, then so is $B_\al\ti\phi$,
with $\Al B_\al \ti\phi = B^\om \Al\ti\phi$.
If moreover $\Al\ti\phi \in z^{-2}\C[[z\ii]]$ (\ie $c=\hat\chi(0)=0$
in~\eqref{eqsimplsing}), then also $E_\be\ti\phi$ is $(\om,\Ga)$-simple
and $\Al E_\be \ti\phi - E \Al\ti\phi$ is a constant.
\end{Lm}


\begin{proof}[Proof of Lemma~\ref{LmCommutAlBal}]
The first identity results from Lemma~\ref{LmCommutAl} and $\ee^{-\om b} c_\al = \ee^{-\om b_*}$.
For the second one, in view of~\eqref{eqsimplsing}, the analytic
continuation of $\gB E_\be\ti\phi$ is
$ \cont_\Ga \hat E_\be\hat\phi(\om+\ze) =
\frac{\hat\chi(\ze)}{\ee^\ze-1} \frac{\log\ze}{2\pi\I} + \frac{\hat R(\ze)}{\ee^\ze-1}$,
hence $\Al E_\be \ti\phi = E\gB\ii\hat\chi + 2\pi\I\hat R(0)$.
\end{proof}

%
We finish the proof of Theorem~\ref{ThBE} by estimating $\hat\Phi_k \defeq \gB\ti\Phi_k$ by
means of the recursive formula
$\hat\Phi_k = \hat E_\be \hat B_\al \hat\Phi_{k-1}$.
%
%

%
\begin{Lm}   \label{LmEstimEbeBal}
For $0<\eps<1$, $L>1$, let $\tRL$ denote the set of all naturally
parametrised paths $\ga \col (0,\ell(\ga)] \to \Clog$, where
$\ell(\ga) < L$ is the length of~$\ga$,
for which there exists $\th\in[-\pi,3\pi]$ such that
$\ga(s) = s\,\ee^{\I\th}$ for $s\leq\eps$ and
$\dist\big( \ga(s), 2\pi\I\Z \big) > \eps$ for $s\geq\eps$.
%
%
For any integer $\ell\geq0$, let $\VL\ell$ denote the space of
all $\hat\phi \in \ze^{\be+\ell}\C\{\ze\}$ admitting analytic
continuation along the paths of~$\tRL$ and such that
$\NL{\hat\phi}{\ell} \defeq \sup \big\{\,
s^{-\RE\be-\ell} \abs{\cont_\ga \hat\phi\big( \ga(s) \big)},
\;\text{for $\ga\in\tRL$ and $s\in \big( 0,\ell(\ga) \big]$}
\,\big\}$ is finite.
Then the operator $\hat E_\be\hat B_\al$ leaves invariant the spaces $\VL\ell$ and,
for each $\eps$ and~$L$, there exists $M(\eps,L)>0$ such that
\[
\NL{\hat E_\be\hat B_\al\hat\phi}{\ell} \leq
\tfrac{M(\eps,L)}{\RE\be+\ell+1} \NL{\hat\phi}{\ell}
\quad \text{for all $\hat\phi\in\VL\ell$ and $\ell\geq0$.}
\]
\end{Lm}


\begin{proof}[Proof of Lemma~\ref{LmEstimEbeBal}]
One can find $\ka=\ka(\eps,L) > 0$ such that, for each $\ga\in\tRL$
and $s\in(0,\ell(\ga)]$,
$\abs{\ga(s)} \geq \ka s$
(because $\arg\ga(s)$ is uniformly bounded).
Let
$M_0=M_0(\eps,L) \defeq \sup \big\{\,
\abs*{\frac{\ze}{\ee^\ze-1}}, \;\text{for}\;
  \abs\ze < L,\; \dist\big( \ze, 2\pi\I\Z^* \big) > \eps
\,\big\}$.
From~\eqref{eqBalKal}, we get
$\NL{\hat B_\al\hat\phi}{\ell+1} \leq
\tfrac{N_L}{\RE\be+\ell+1} \NL{\hat\phi}{\ell}$
with $N_L \defeq \sup_{\abs{\xi}, \abs{\ze}<L} \{\,
\abs{K(\xi,\ze)} \,\}$,
and
$\NL{\hat E_\be\hat B_\al\hat\phi}{\ell} \leq
\tfrac{N_L M_0}{\ka(\RE\be+\ell+1)} \NL{\hat\phi}{\ell}$.
\end{proof}
%

We have $\abs*{ \frac{\al}{\al+N+1} } < 1$ because $\RE\al > - \frac{N+1}{2}$.
Let us choose an integer $d\geq0$ so that
$\tfrac{M(\eps,L)}{\RE\be+d+1} \leq
\La \defeq \max\Big\{ \abs*{ \frac{\al}{\al+N+1} }, \demi \Big\}$.
For any $\hat\phi \in\ze^\be\C\{\ze\}$, we use the notation
\[
\hat\phi = \bd{\hat\phi} + \cd{\hat\phi},
\qquad
\bd{\hat\phi} \in \gE_d \defeq
\Span(\ze^\be, \ze^{\be+1}, \ldots,
\ze^{\be+d-1} ),
\quad \cd{\hat\phi} \in \ze^{\be+d}\C\{\ze\}.
\]
%
%
Now $\hat\phi \in \ze^{\be+d}\C\{\ze\} \,\Rightarrow\,
\hat B_\al\hat\phi = \al 1*\hat\phi + O(\ze^{\be+d+2})
\,\Rightarrow\,
\hat E_\be\hat B_\al\hat\phi =
\frac{\al}{\ze} 1*\hat\phi + O(\ze^{\be+d+1})
\in \ze^{\be+d}\C\{\ze\}$, thus
\[
\bd{\hat\Phi_k} =  A \bd{\hat\Phi_{k-1}},
\qquad
\cd{\hat\Phi_k} = \cd{\hat E_\be\hat B_\al \bd{\hat\Phi_{k-1}}} +
\hat E_\be\hat B_\al \cd{\hat\Phi_{k-1}},
\]
where $A\in\End\gE_d$ is defined by $A\hat\phi \defeq \bd{\hat E_\be\hat B_\al \hat\phi}$
and has a triangular matrix in the basis $(\ze^\be, \ze^{\be+1}, \ldots,
\ze^{\be+d-1} )$, with
$A \ze^{\be+\ell} = \la_\ell \ze^{\be+\ell}\big(1+O(\ze)\big)$,
$\la_\ell \defeq \tfrac{\al}{\be+\ell+1}$,
$\ell = 0, \ldots, d-1$.
By the choice of $\be$, the eigenvalues have modulus
$\abs{\la_\ell} \leq \La$,
hence $\norm{ \bd{\hat\Phi_k} } = O(\La^k)$ for any norm on~$\gE_d$.
Since
$\hat\phi\in \gE_d\mapsto\cd{\hat E_\be \hat B_\al\hat\phi} \in \VL d$
is linear, we can find $C_0'$ such that
$\NL{ \cd{\hat E_\be \hat B_\al\bd{\hat\Phi_{k-1}}} }{d} \leq
C'_0 \La^k$ for all $k\geq1$, hence
$\NL{ \cd{\hat\Phi_k} }{d} \leq  C'_0 \La^k + \La \NL{ \cd{\hat\Phi_{k-1}} }{d}$.
We thus get
$\NL{ \cd{\hat\Phi_k} }{d} \leq \big(
\NL{ \cd{\hat\Phi_0} }{d} + C'_0 k
\big) \La^k$ for all $k\geq0$ and,
increasing slightly~$\La$, the desired bounds for $\abs*{ \cont_\ga \hat\Phi_k }$.


By uniform convergence, we can now define a function $\hat\Phi = \sum_{k\geq0}
\hat\Phi_k \in \ze^\be\C\{\ze\}$, whose inverse formal Borel transform
$\ti\Phi \in z^{-\be-1}\C[[z\ii]]$ satisfies $\ti\Phi = E_\be b_\al +
E_\be B_\al \ti\Phi$, whence
\[
C_{\id-1} \ti\Phi = c_\al C_{\id+b} \ti\Phi + b_\al.
\]
Multiplying both sides by $z^\al(1-z\ii)^\al$, we find that
$\ti\psi \defeq z^\al\ti\Phi \in z^{-N-1}\C[[z\ii]]$ satisfies
$C_{\id-1} \ti\psi = C_{\id+b} \ti\psi + b_N$,
an equation of which $\{\ti\ph\}_N$ is the unique solution in $\zcz$,
hence $\ti\Phi = z^{-\al}\{\ti\ph\}_N$.


The estimates for the $\hat\Phi_k$'s imply similar estimates for the monodromies
$\hat\chi_k \defeq \hAl\ti\Phi_k$, and hence for the residues
$S^\Ga_{\om,k}$  of the functions
$\cont_\ga\hat\Phi_k(\om+\ze) - \hat\chi(\ze) \frac{\log\ze}{2\pi\I}$.
Equation~\eqref{eqBEb} follows by uniform convergence.
\end{proof}


%




\smallskip

\noindent{\footnotesize \emph{Acknowledgements.}
  The authors are grateful to F.~Fauvet and M.~Yampolsky for fruitful
  discussions and encouragements.
  The first author expresses thanks to the Centro di Ricerca
  Matematica Ennio De Giorgi and Fibonacci Laboratory for
  hospitality.
  The second author acknowledges the support of the French National
  Research Agency under the reference ANR-12-BS01-0017.  }




\vspace{.9cm}

\noindent
Artem Dudko:
Institute for Mathematical Sciences,
University of Stony Brook, NY, USA.

\bigskip

\noindent
David Sauzin:
CNRS UMI 3483 - Laboratoire Fibonacci,
Centro di Ricerca Matematica Ennio De Giorgi,
Scuola Normale Superiore di Pisa, Italy.


\end{document}